\documentclass[reqno]{amsart}
\usepackage{etoolbox}

\makeatletter
\let\old@tocline\@tocline
\let\section@tocline\@tocline

\newcommand{\subsection@dotsep}{4.5}
\newcommand{\subsubsection@dotsep}{4.5}
\patchcmd{\@tocline}
  {\hfil}
  {\nobreak
     \leaders\hbox{$\m@th
        \mkern \subsection@dotsep mu\hbox{.}\mkern \subsection@dotsep mu$}\hfill
     \nobreak}{}{}
\let\subsection@tocline\@tocline
\let\@tocline\old@tocline

\patchcmd{\@tocline}
  {\hfil}
  {\nobreak
     \leaders\hbox{$\m@th
        \mkern \subsubsection@dotsep mu\hbox{.}\mkern \subsubsection@dotsep mu$}\hfill
     \nobreak}{}{}
\let\subsubsection@tocline\@tocline
\let\@tocline\old@tocline

\let\old@l@subsection\l@subsection
\let\old@l@subsubsection\l@subsubsection

\def\@tocwriteb#1#2#3{%
  \begingroup
    \@xp\def\csname #2@tocline\endcsname##1##2##3##4##5##6{%
      \ifnum##1>\c@tocdepth
      \else \sbox\z@{##5\let\indentlabel\@tochangmeasure##6}\fi}%
    \csname l@#2\endcsname{#1{\csname#2name\endcsname}{\@secnumber}{}}%
  \endgroup
  \addcontentsline{toc}{#2}%
    {\protect#1{\csname#2name\endcsname}{\@secnumber}{#3}}}%

\newlength{\@tocsectionindent}
\newlength{\@tocsubsectionindent}
\newlength{\@tocsubsubsectionindent}
\newlength{\@tocsectionnumwidth}
\newlength{\@tocsubsectionnumwidth}
\newlength{\@tocsubsubsectionnumwidth}
\newcommand{\settocsectionnumwidth}[1]{\setlength{\@tocsectionnumwidth}{#1}}
\newcommand{\settocsubsectionnumwidth}[1]{\setlength{\@tocsubsectionnumwidth}{#1}}
\newcommand{\settocsubsubsectionnumwidth}[1]{\setlength{\@tocsubsubsectionnumwidth}{#1}}
\newcommand{\settocsectionindent}[1]{\setlength{\@tocsectionindent}{#1}}
\newcommand{\settocsubsectionindent}[1]{\setlength{\@tocsubsectionindent}{#1}}
\newcommand{\settocsubsubsectionindent}[1]{\setlength{\@tocsubsubsectionindent}{#1}}

\renewcommand{\l@section}{\section@tocline{1}{\@tocsectionvskip}{\@tocsectionindent}{\@tocsectionnumwidth}{\@tocsectionformat}}%
\renewcommand{\l@subsection}{\subsection@tocline{1}{\@tocsubsectionvskip}{\@tocsubsectionindent}{\@tocsubsectionnumwidth}{\@tocsubsectionformat}}%
\renewcommand{\l@subsubsection}{\subsubsection@tocline{1}{\@tocsubsubsectionvskip}{\@tocsubsubsectionindent}{\@tocsubsubsectionnumwidth}{\@tocsubsubsectionformat}}%
\newcommand{\@tocsectionformat}{}
\newcommand{\@tocsubsectionformat}{}
\newcommand{\@tocsubsubsectionformat}{}
\expandafter\def\csname toc@1format\endcsname{\@tocsectionformat}
\expandafter\def\csname toc@2format\endcsname{\@tocsubsectionformat}
\expandafter\def\csname toc@3format\endcsname{\@tocsubsubsectionformat}
\newcommand{\settocsectionformat}[1]{\renewcommand{\@tocsectionformat}{#1}}
\newcommand{\settocsubsectionformat}[1]{\renewcommand{\@tocsubsectionformat}{#1}}
\newcommand{\settocsubsubsectionformat}[1]{\renewcommand{\@tocsubsubsectionformat}{#1}}
\newlength{\@tocsectionvskip}
\newcommand{\settocsectionvskip}[1]{\setlength{\@tocsectionvskip}{#1}}
\newlength{\@tocsubsectionvskip}
\newcommand{\settocsubsectionvskip}[1]{\setlength{\@tocsubsectionvskip}{#1}}
\newlength{\@tocsubsubsectionvskip}
\newcommand{\settocsubsubsectionvskip}[1]{\setlength{\@tocsubsubsectionvskip}{#1}}

\patchcmd{\tocsection}{\indentlabel}{\makebox[\@tocsectionnumwidth][l]}{}{}
\patchcmd{\tocsubsection}{\indentlabel}{\makebox[\@tocsubsectionnumwidth][l]}{}{}
\patchcmd{\tocsubsubsection}{\indentlabel}{\makebox[\@tocsubsubsectionnumwidth][l]}{}{}

\newcommand{\@sectypepnumformat}{}
\renewcommand{\contentsline}[1]{%
  \expandafter\let\expandafter\@sectypepnumformat\csname @toc#1pnumformat\endcsname%
  \csname l@#1\endcsname}
\newcommand{\@tocsectionpnumformat}{}
\newcommand{\@tocsubsectionpnumformat}{}
\newcommand{\@tocsubsubsectionpnumformat}{}
\newcommand{\setsectionpnumformat}[1]{\renewcommand{\@tocsectionpnumformat}{#1}}
\newcommand{\setsubsectionpnumformat}[1]{\renewcommand{\@tocsubsectionpnumformat}{#1}}
\newcommand{\setsubsubsectionpnumformat}[1]{\renewcommand{\@tocsubsubsectionpnumformat}{#1}}
\renewcommand{\@tocpagenum}[1]{%
  \hfill {\mdseries\@sectypepnumformat #1}}

\let\oldappendix\appendix
\renewcommand{\appendix}{%
  \leavevmode\oldappendix%
  \addtocontents{toc}{%
    \protect\settowidth{\protect\@tocsectionnumwidth}{\protect\@tocsectionformat\sectionname\space}%
    \protect\addtolength{\protect\@tocsectionnumwidth}{2em}}%
}
\makeatother

\makeatletter
\settocsectionnumwidth{2em}
\settocsubsectionnumwidth{2.5em}
\settocsubsubsectionnumwidth{3em}
\settocsectionindent{1pc}%
\settocsubsectionindent{\dimexpr\@tocsectionindent+\@tocsectionnumwidth}%
\settocsubsubsectionindent{\dimexpr\@tocsubsectionindent+\@tocsubsectionnumwidth}%
\makeatother

\settocsectionvskip{2pt}
\settocsubsectionvskip{0pt}
\settocsubsubsectionvskip{0pt}

\settocsectionformat{\bfseries}
\settocsubsectionformat{\mdseries}
\settocsubsubsectionformat{\mdseries}
\setsectionpnumformat{\bfseries}
\setsubsectionpnumformat{\mdseries}
\setsubsubsectionpnumformat{\mdseries}

\let\oldtableofcontents\tableofcontents
\renewcommand{\tableofcontents}{%
  \vspace*{-5\linespacing}
  \oldtableofcontents}

\makeatletter
\let\origsection=\section \def\section{\@ifstar{\origsection*}{\mysection}} 
\def\mysection{\@startsection{section}{1}\z@{.7\linespacing\@plus\linespacing}{.5\linespacing}{\normalfont\scshape\centering\S}}
\makeatother  

\usepackage{amssymb}
\usepackage{amsmath}
\usepackage{amsthm}
\usepackage{letterswitharrows}
\usepackage{mathrsfs}
\usepackage{graphicx}
\usepackage{latexsym}
\usepackage{stmaryrd}
\usepackage[shortlabels]{enumitem}
\usepackage{moreenum}
\usepackage[bookmarks,hyperfootnotes=false, psdextra=true]{hyperref}
\usepackage{multirow}
\usepackage{xcolor}
\usepackage{caption}
\usepackage{subcaption}

\colorlet{darkishRed}{red!60!black}
\colorlet{darkishBlue}{blue!60!black}
\colorlet{darkishGreen}{green!50!black}
\colorlet{darkerishGreen}{green!30!black}
\colorlet{lightishGreen}{green!70!black}
\hypersetup{
    draft = false,
    bookmarksopen=true,
    colorlinks,
    linkcolor={darkishBlue},
    citecolor={darkishBlue},
    urlcolor={darkishBlue}
}
\usepackage[nameinlink, capitalise, noabbrev]{cleveref}
\usepackage{thmtools}
\usepackage{nameref}
\crefformat{enumi}{#2#1#3}
\crefformat{equation}{#2(#1)#3}
\crefrangeformat{section}{Sections~#3#1#4--#5#2#6}
\crefname{mainresult}{Theorem}{Theorems}
\usepackage{doi}

\usepackage{tikz}
\usepackage{tikz-cd}
\usetikzlibrary{calc,through,intersections,arrows, trees, positioning, decorations.pathmorphing, cd}
\usepackage{comment}
\usepackage{mathtools}
\usepackage{nccmath}
\usepackage{pifont}
\usepackage[utf8]{inputenc}
\usepackage[T1]{fontenc}
\usepackage{lmodern}
\usepackage[babel]{microtype}
\usepackage[english]{babel}
\usepackage{relsize}

\newcommand{\COMMENT}[1]{{}}

\usepackage{geometry}
\renewcommand{\leq}{\leqslant}
\renewcommand{\geq}{\geqslant}

\let\rho=\varrho
\let\phi=\varphi

\newcommand{\defn}[1]{{\color{darkishGreen}{\emph{#1}}}}

\makeatletter

\def\calCommandfactory#1{%
   \expandafter\def\csname c#1\endcsname{\mathcal{#1}}}
\def\frakCommandfactory#1{%
   \expandafter\def\csname frak#1\endcsname{\mathfrak{#1}}}

\newcounter{ctr}
\loop
  \stepcounter{ctr}
  \edef\X{\@Alph\c@ctr}
  \expandafter\calCommandfactory\X
  \expandafter\frakCommandfactory\X
  \edef\Y{\@alph\c@ctr}
  \expandafter\frakCommandfactory\Y
\ifnum\thectr<26
\repeat

\newcommand{\bH}{\mathbf{H}}

\setenumerate{label={\normalfont (\roman*)}}

\newtheorem{theorem}{Theorem}[section] 

\newtheorem{lemma}[theorem]{Lemma}

\newtheorem{conjecture}[theorem]{Conjecture}

\newtheorem{mainresult}{Theorem}

\newtheorem{claim}{Claim}
\crefname{claim}{Claim}{Claims}
\AtEndEnvironment{proof}{\setcounter{claim}{0}}

\theoremstyle{definition}

\newtheorem*{definition*}{Definition}

\theoremstyle{remark}

\usepackage{etoolbox}
\newbool{pdfBool}
\boolfalse{pdfBool} 

\newcommand{\pdfOrNot}[2]{\ifbool{pdfBool}{{#1}}{{#2}}}

\usepackage{svg}

\def\lqedsymbol{\ifmmode$\lrcorner$\else{\unskip\nobreak\hfil
		\penalty50\hskip1em\null\nobreak\hfil$\rule{1.2ex}{1.2ex}$
		\parfillskip=0pt\finalhyphendemerits=0\endgraf}\fi}

\newenvironment{claimproof}[1][\proofname]
{%
	\proof[#1]%
}
{%
	\endproof%
}

\usepackage{tikz}
\usetikzlibrary{positioning,shapes,calc}

\usepackage{csquotes}
\usepackage[backend=biber, style=alphabetic, maxnames=99, maxalphanames=99, giveninits=true, sortcites=true]{biblatex}
\begin{document}

\title{Locally interval graphs are circular-arc graphs}

\author[T.\ Abrishami \and S.\ Albrechtsen \and N.\ Bowler \and P.\ Knappe\and J.\ Nickel]{Tara Abrishami, Sandra Albrechtsen, Nathan Bowler, Paul Knappe, Jana Katharina Nickel}

\address{Stanford University, Department of Mathematics}
\email{tara.abrishami@stanford.edu}

\address{Leipzig University, Institute of Mathematics}
\email{sandra@albrechtsen-mail.de}

\address{University of Hamburg, Department of Mathematics}
\email{\{nathan.bowler, paul.knappe, jana.nickel-2\}@uni-hamburg.de}

\thanks{T.A. is supported by the National Science Foundation Award Number DMS-2303251 and the Alexander von Humboldt Foundation. S.A. is supported by the Alexander von Humboldt Foundation in the framework of the Alexander von Humboldt Professorship of Daniel Král' endowed by the Federal Ministry of Education and Research. P.K. is supported by a doctoral scholarship of the Studienstiftung des deutschen Volkes.}


\begin{abstract}
    Circular-arc graphs are graphs that can be represented as intersection graphs of subpaths of a cycle. Interval graphs are graphs that can be represented as intersection graphs of subpaths of a path. Since cycles are locally paths, every circular-arc graph is locally interval. In this paper, we prove that the converse holds as well: every locally interval graph is a circular-arc graph. This result and its proofs are connected to a recent broader study of structural local-global theory and build on previous work on locally chordal graphs. 
\end{abstract}

\maketitle

\section{Introduction}

If every vertex in a connected graph has degree at most two, then the graph is a path or a cycle. This observation can be extended: if the radius-$r$ ball around every vertex of a connected graph $G$ is a path, then the graph $G$ is either a path or a cycle of length greater than $2r$. Since small-radius balls in a long cycle are all paths, this is a characterization: a graph is a path or a long cycle if and only if every small-radius ball of the graph is a path.  

In this paper, we take this simple idea --- that paths and long cycles are precisely the shapes that are locally paths everywhere --- and apply it to a richer class of graphs. Roughly, we prove that the graphs that are locally ``path-like'' everywhere are precisely the graphs that are globally ``path-like'' or ``long cycle-like.'' To make this precise, let us give some definitions. A graph $G$ is an \defn{interval graph} if there is a path $P$ and a mapping $v \mapsto P_v$ from vertices $v$ of $G$ to subpaths $P_v$ of $P$ such that there is an edge $uv$ in $G$ if and only if $P_u$ and $P_v$ intersect in at least one vertex. We call such a mapping $v \mapsto P_v$ an \defn{interval representation (of $G$ over $P$)}. Interval graphs are a special case of more general graphs, region intersection graphs. A graph $G$ is a \defn{region intersection graph over $H$} if there is a mapping $v \mapsto H_v$ from vertices $v$ of $G$ to connected subgraphs (i.e. \defn{regions}) $H_v$ of $H$ such that $uv$ is an edge of $G$ if and only if $H_u$ and $H_v$ intersect in at least one vertex. The mapping $v \mapsto H_v$ is called a \defn{region representation (of $G$ over $H$)}. In this language, interval graphs are the class of graphs which admit region representations over paths. 

Because of their characterization as precisely the graphs which admit region representations over paths, interval graphs can be reasonably understood as graphs which are ``path-like.'' We are interested in graphs which are ``locally path-like,'' so let us next understand what we mean by ``locally.'' Given a graph $G$ and a vertex $v$ of $G$, the \defn{ball of radius $r/2$ around $v$}, denoted \defn{$B_G(v, r/2)$}, is the subgraph of $G$ with vertex set consisting of the vertices of distance at most $\lfloor r/2 \rfloor$ from $v$, and edge set consisting of the edges $xy$ of $G$ such that $d(v, x) + d(v, y) < r$. A graph $G$ is \defn{$r$-locally interval} for an integer $r \geq 0$ if all its $r/2$-balls $B_G(v, r/2)$ are interval graphs. 

What kind of graphs might be locally interval? Since every interval graph has a region representation over a path, it seems sensible to look for graphs that have a region representation over a graph that is ``locally a path.'' Therefore, a good guess seems to be that locally interval graphs are exactly those graphs which have region representations over a (long) cycle. These form a well-known graph class, the \emph{circular-arc} graphs. A graph $G$ is \defn{circular-arc} if it has a region representation $v \mapsto C_v$  over a cycle $C$ such that the regions $C_v$ are paths.
We refer to such region representations over cycles as \defn{circular-arc representations}. A circular-arc representation $v \mapsto C_v$ of a graph $G$ over a cycle $C$ is \defn{$r$-acyclic} if for every vertex-set $X \subseteq V(G)$ of size at most $r$, the subgraph $\bigcup_{x \in X} C_x$ of $C$ is acyclic. We call a graph $G$ an \defn{$r$-acyclic} circular-arc graph if $G$ has an $r$-acyclic circular-arc representation. 
Note that $\infty$-acyclic circular-arc graphs are precisely interval graphs, because $\bigcup_{v \in V(G)} C_v$ is then a path.

With these definitions, the following intuitive statement, conjectured by Pawe\l\ Rz\k{a}\.{z}ewski in private communication, seems morally true: 
\begin{conjecture}[Rz\k{a}\.{z}ewski 2024]\label{conj:loc-I-is-CAG}
    Locally interval graphs are exactly circular-arc graphs.
\end{conjecture}

We make this conjecture precise via our definitions above. Specifically, we ask: are the $r$-locally interval graphs precisely the $r$-acyclic circular-arc graphs? It turns out that this is false for $r \leq 3$, but we prove that it is true for $r \geq 4$:

\begin{theorem}
    Let $G$ be a finite connected graph and $r \geq 4$ an integer. Then, $G$ is $r$-locally interval if and only if $G$ is $r$-acyclic circular-arc.
\end{theorem}

Circular-arc graphs are widely studied. Recently, Cao, Grippo, and Safe \cite{ForbIndSubNHCAG} gave a forbidden induced subgraph characterization of 3-acyclic circular-arc graphs; see \cref{thm:Char-Forb-Ind-Sub-N-H-CA}. 
Building on their work, we give a forbidden induced subgraph characterization of $r$-locally interval graphs for $r \geq 3$.

\subsection{Connections to local-global theory}
In addition to addressing an intuitive and visually appealing idea, our study of locally interval graphs connects to a broader study of local-global graph theory. 
This line of study began with the work of Diestel, Jacobs, Knappe, and Kurkofka \cite{canonicalGD}, which introduced an idea called the \emph{$r$-local covering} to study the ``$r$-local'' structure of a graph. Given a graph $G$ and an integer $r \geq 1$, the \defn{$r$-local covering of $G$}, denoted \defn{$p_r : G_r \to G$}, is the universal covering of $G$ that preserves its $r/2$-balls. (See \cref{subsec:loc-cover} for the formal definition.) The graph $G_r$ is called the \defn{$r$-local cover of $G$}. 

In \cites{LocallyChordal, localGlobalChordal}, the first and fourth authors and Kobler studied \emph{locally chordal graphs}. A graph is \defn{chordal} if each of its cycles of length at least four has a chord. Equivalently, a graph $G$ is chordal if and only if $G$ admits a region representation over a tree. Since interval graphs are the graphs that admit region representations over paths, interval graphs are a subclass of chordal graphs. A graph is \defn{$r$-locally chordal} if each of its $r/2$-balls is chordal. 

The following theorem characterizes locally chordal graphs by their $r$-local covers, by forbidden induced subgraphs, and by region representations. Formal definitions for statements (iii) and (iv) below can be found in \cref{subsec:locally-chordal}.

\begin{theorem}[\cite{LocallyChordal}, Theorem 1 and \cite{localGlobalChordal}, Theorem 1]
\label{thm:locally-chordal}
    Let $G$ be a graph and let $r \geq 3$ be an integer. The following are equivalent: 
    \begin{enumerate}
        \item $G$ is $r$-locally chordal. 
        \item The $r$-local cover $G_r$ of $G$ is chordal.
        \item $G$ is $r$-chordal and wheel-free. 
        \item $G$ is an $r$-acyclic region intersection graph. 
    \end{enumerate}
\end{theorem}

Our main result in this paper is an analogous theorem for locally interval graphs, giving characterizations by their local covers, by forbidden induced subgraphs, and as circular-arc graphs.  

\begin{mainresult}\label{main:thm:Char-Forb-Ind-Sub-r-ACA}
    Let $G$ be a finite connected graph and $r \geq 4$ an integer.
    Then the following are equivalent:
    \begin{enumerate}
        \item \label{item:r-ACA:locInt} $G$ is $r$-locally interval.
        \item \label{item:r-ACA:locCov} The $r$-local cover $G_r$ of $G$ is interval.
        \item \label{item:r-ACA:ForbInd} $G$ is $r$-chordal, wheel-free, and contains none of the graphs depicted in \cref{fig:Ford-Ind-Sub-Interval} as an induced subgraph. If $r = 4$, then $G$ also contains none of the graphs depicted in \cref{fig:Forb-Ind-Sub-4-ACA} as an induced subgraph.
        \item \label{item:r-ACA:r-ACA} $G$ is $r$-acyclic circular-arc.
    \end{enumerate}
\end{mainresult}

\section{Background}

\subsection{Locally chordal graphs}\label{subsec:locally-chordal} 
We begin by giving the definitions necessary to understand statements (iii) and (iv) of \cref{thm:locally-chordal}. A graph is \defn{$r$-chordal} if each of its cycles of length at least four and at most $r$ has a chord. Equivalently, a graph is $r$-chordal if it does not contain an induced cycle of length $\ell$ with $4 \leq \ell \leq r$. A \defn{wheel $W_n$} for $n \geq 4$ is a graph consisting of a cycle $C_n$ of length $n$ and an additional vertex $v$ that is adjacent to every vertex of $C_n$. A graph is \defn{wheel-free} if it does not contain a wheel $W_n$ for $n \geq 4$ as an induced subgraph. Observe that every chordal graph is $r$-chordal and wheel-free. The equivalence between (i) and (iii) in \cref{thm:locally-chordal} states that being $r$-chordal and wheel-free is equivalent to being $r$-locally chordal. 

 A \defn{region} of a graph $G$ is a connected subgraph of $G$. A family of regions $\bH$ of a graph $H$ is \defn{Helly} if it satisfies the \defn{Helly property}: every finite subset of pairwise intersecting regions of $\bH$ has nonempty intersection. A region representation $v \mapsto H_v$ of a graph $G$ over a graph $H$ is \defn{$r$-acyclic} if for every set $X \subseteq V(G)$ of size at most $r$, the subgraph $\bigcup_{x \in X} H_x$ of $H$ is acyclic. A graph $G$ is an \defn{$r$-acyclic region intersection graph} if $G$ admits an $r$-acyclic region representation.

Next, we give two statements from \cite{localGlobalChordal} that we will use in the proof of our main result. 

\begin{lemma}[\cite{localGlobalChordal}, Lemma 3.6]\label{lem:AK3.6}
    Let $\bH$ be a family of regions of a (possibly infinite) graph $H$. Then $\bH$ is 3-acyclic if and only if $\bH$ is 2-acyclic and Helly. 
\end{lemma}

\begin{theorem}[\cite{localGlobalChordal}, Theorem 5]
\label{thm:AK-thm5}
    Let $v \mapsto H_v$ be a 3-acyclic region representation of a (possibly infinite) graph $G$ over a (possibly infinite) graph $H$ and let $r > 3$ be an integer. Then, $G$ is $r$-locally chordal if and only if $v \mapsto H_v$ is $r$-acyclic. 
\end{theorem}

\subsection{Interval graphs}
Boland \& Lekkeikerker~\cite{ForbIndSubInterval} characterized finite interval graphs by forbidden induced subgraphs:

\begin{theorem}[Boland \& Lekkeikerker  \cite{ForbIndSubInterval}, II]\label{thm:Char-Forb-Ind-Sub-Interval}
    Let $G$ be a (possibly infinite)\footnote{Even though Boland and Lekkeikerker prove \cref{thm:Char-Forb-Ind-Sub-Interval} only for finite graphs, Halin~\cite[(8)]{halininterval} showed that an infinite graph is interval if and only if its finite induced subgraphs are interval, which implies \cref{thm:Char-Forb-Ind-Sub-Interval} for infinite graphs.} graph.
    Then the following are equivalent:
    \begin{enumerate}
        \item $G$ is interval.
        \item $G$ is chordal and contains none of the graphs depicted in \cref{fig:Ford-Ind-Sub-Interval} as an induced subgraph.
    \end{enumerate}
\end{theorem}

\begin{figure}[ht]
    \centering
    \begin{subfigure}[b]{0.225\textwidth}
         \centering
         \includegraphics[width=0.6\textwidth]{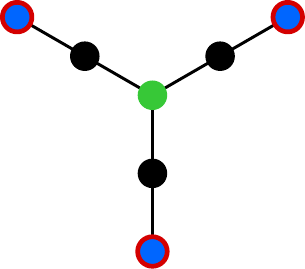}
         \caption{Long claw}
         \label{fig:LongClaw}
     \end{subfigure}
     \hfill
     \begin{subfigure}[b]{0.225\textwidth}
         \centering
         \includegraphics[width=0.75\textwidth]{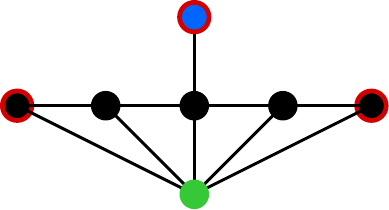}
         \caption{Whipping top}
         \label{fig:WhippingTop}
     \end{subfigure}
     \hfill
     \begin{subfigure}[b]{0.225\textwidth}
         \centering
         \includegraphics[width=0.75\textwidth]{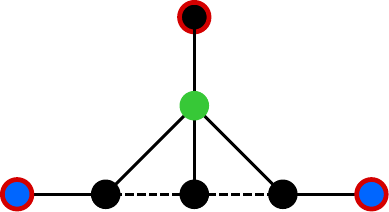}
         \caption{$\dag$}
         \label{fig:dag}
     \end{subfigure}
     \hfill
     \begin{subfigure}[b]{0.225\textwidth}
         \centering
         \includegraphics[width=0.75\textwidth]{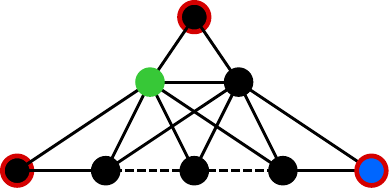}
         \caption{$\ddag$}
         \label{fig:ddag}
     \end{subfigure}
    \caption{Chordal minimal forbidden induced subgraphs of interval graphs. The dotted lines in \cref{fig:dag} and \cref{fig:ddag} represent paths of arbitrary length, except one of the dotted lines in \cref{fig:dag} must be of length at least one.}
    \label{fig:Ford-Ind-Sub-Interval}
\end{figure}

As a sanity check, we show that every interval graph is indeed locally interval. (Observe that this is not immediate from the definition. Every induced subgraph of an interval graph is an interval graph, but when $r$ is even, the $r/2$-balls of a graph are not necessarily induced subgraphs.) 

\begin{lemma}\label{lem:IntervalIsLocallyInterval}
    Every (possibly infinite) interval graph is $r$-locally interval for every integer $r \geq 0$.
\end{lemma}

\begin{proof}
    Let $G$ be an interval graph and $r \geq 0$ an integer.
    Since interval graphs are chordal, the graph $G$ is chordal, and so in particular, $G$ is $r$-chordal and wheel-free. Therefore, by \cref{thm:locally-chordal}, $G$ is $r$-locally chordal. 
    Suppose for a contradiction that an $r/2$-ball $B$ in $G$ around a vertex $v$ contains a graph $F$ depicted in \cref{fig:Ford-Ind-Sub-Interval} as induced subgraph.
    If $r$ is odd, then $B$ is an induced subgraph of $G$, and thus the non-interval graph $F$ is also an induced subgraph of $G$.
    Therefore, assume that $r$ is even, and set $d \coloneqq r/2$.
    
    Since $F$ is not an induced subgraph of $G$, some two vertices of $F$ have to lie in $N^d(v)$, where $N^d(v)$ denotes the set of vertices of distance exactly $d$ from $v$.
    By \cref{thm:locally-chordal}, for every vertex $u \in N^d(v)$, its neighborhood $N_B(u)$ is a clique in $B$.
    Thus, only the red circled vertices in \cref{fig:Ford-Ind-Sub-Interval} may be contained in $N^d(v)$.
    It is a simple check that every potential edge of $G$ between some pair of red circled vertices yields an induced cycle of length at least $4$ in $G$, which contradicts that $G$ is chordal.
\end{proof}

\subsection{The local covering}\label{subsec:loc-cover}
Finally, we introduce the $r$-local covering. A \defn{covering} of a connected graph $G$ is a surjective mapping $p: C \to G$ from a graph $C$ to $G$ such that for every vertex $c$ of $C$, the mapping $p$ restricts to an isomorphism from the edges incident to $c$ (in $C$) to the edges incident to $p(c)$ (in $G$). Given an integer $r \geq 1$, a covering $p: C \to G$ is \defn{$r/2$-ball-preserving} if $p$ restricts to an isomorphism from $B_G(c, r/2)$ to $B_G(p(c), r/2)$ for every vertex $c$ of $C$. A covering $p: C \to G$ of $G$ is \defn{universal} among a set of coverings $\mathcal{C}$ of $G$ if for every covering $q: D \to G$ in $\mathcal{C}$ there exists a covering $s: C \to D$ such that $p = q \circ s$. 

The formal definition of the $r$-local covering is given in \cite[\S 4.2]{canonicalGD}. For this paper, we need only its equivalent characterization given by \cite[Lemmas 4.2, 4.3, and 4.4]{canonicalGD}. Given a connected graph $G$ and an integer $r \geq 1$, the \defn{$r$-local covering $p_r: G_r \to G$} is the universal covering among the coverings of $G$ that preserve $r/2$-balls. The graph $G_r$ is the \defn{$r$-local cover} of $G$. Given $G$ and $r\geq 1$, both the $r$-local covering $p_r: G_r \to G$ and the $r$-local cover $G_r$ are unique. 

\section{Proof of Theorem 1}

In this section, we show that for $r\geq 4$, the $r$-locally interval graphs are precisely the $r$-acyclic circular-arc graphs, i.e.\ \cref{conj:loc-I-is-CAG} holds for $r \geq 4$.

Cao, Grippo and Safe characterized the $3$-acyclic circular-arc graphs by forbidden induced subgraphs (\cref{thm:Char-Forb-Ind-Sub-N-H-CA}).
They refer to $3$-acyclic circular-arc graphs as \emph{normal Helly circular-arc graphs}.
In the literature, $2$-acyclic circular-arc graphs are called \defn{normal}.
These types of circular-arc graphs have often been considered when being additionally \emph{Helly}.
By \cref{lem:AK3.6}, the $3$-acyclic circular-arc graphs are precisely the normal Helly circular-arc graphs.

\begin{theorem}[Cao, Grippo \& Safe~\cite{ForbIndSubNHCAG}, Theorem 1]\label{thm:Char-Forb-Ind-Sub-N-H-CA}
    Let $G$ be a finite graph. Then the following are equivalent:
    \begin{enumerate}
        \item $G$ is $3$-acyclic (equivalently: normal Helly) circular-arc.
        \item $G$ is wheel-free (equivalently: $3$-locally chordal) and contains none of the graphs depicted in \cref{fig:Ford-Ind-Sub-Interval,fig:Ford-Ind-Sub-NHCA}  as an induced subgraph.
    \end{enumerate}
\end{theorem}

\begin{figure}[ht]
    \centering
    \begin{subfigure}[b]{0.1125\textwidth}
        \centering
        \includegraphics[width=\textwidth]{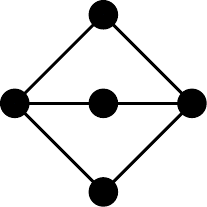}
        \caption{$K_{2,3}$}
        \label{fig:1-1}
    \end{subfigure}
    \hfill
    \begin{subfigure}[b]{0.1125\textwidth}
        \centering
        \includegraphics[width=\textwidth]{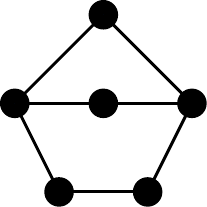}
        \caption{Twin-$C^5$}
        \label{fig:1-2}
    \end{subfigure}
    \hfill
    \begin{subfigure}[b]{0.1125\textwidth}
        \centering
        \includegraphics[width=0.57\textwidth]{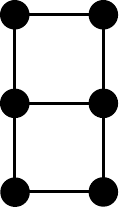}
        \caption{Domino}
        \label{fig:1-3}
    \end{subfigure}
    \hfill
    \begin{subfigure}[b]{0.1125\textwidth}
        \centering
        \includegraphics[width=\textwidth]{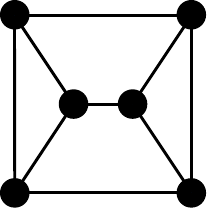}
        \caption{$\bar{C^6}$}
        \label{fig:1-4}
    \end{subfigure}
    \hfill
    \begin{subfigure}[b]{0.1125\textwidth}
        \centering
        \includegraphics[width=\textwidth]{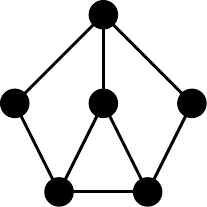}
        \caption{FIS-$1$}
        \label{fig:1-5}
    \end{subfigure}
    \hfill
    \begin{subfigure}[b]{0.1125\textwidth}
        \centering
        \includegraphics[width=\textwidth]{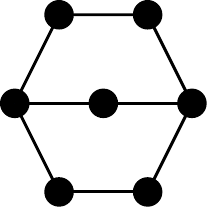}
        \caption{FIS-$2$}
        \label{fig:1-6}
    \end{subfigure}
    \hfill
    \begin{subfigure}[b]{0.125\textwidth}
        \centering
        \includegraphics[width= 0.9 \textwidth]{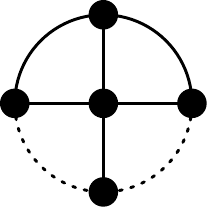}
        \caption{Wheel~$W_n$}
        \label{fig:wheel}
    \end{subfigure}
    \hfill
    \begin{subfigure}[b]{0.1125\textwidth}
        \centering
        \includegraphics[width=\textwidth]{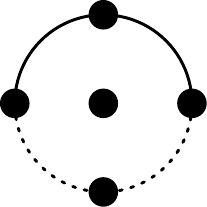}
        \caption{$C_n^*$}
        \label{fig:Cstar}
    \end{subfigure}
    \caption{Non-chordal minimal forbidden induced subgraphs of $3$-acyclic circularc arc graphs, $n \geq 4$.}
    \label{fig:Ford-Ind-Sub-NHCA}
\end{figure}

\begin{figure}[ht]
    \centering
    \centering
    \begin{subfigure}[b]{0.3\textwidth}
        \centering
        \includegraphics[width=0.45\textwidth]{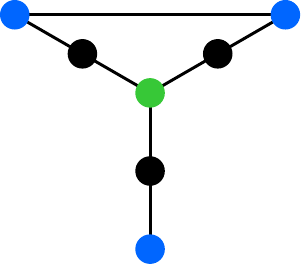}
        \caption{Long Claw plus one edge}
        \label{fig:LC1}
    \end{subfigure}
    \hfill
    \begin{subfigure}[b]{0.3\textwidth}
        \centering
        \includegraphics[width=0.45\textwidth]{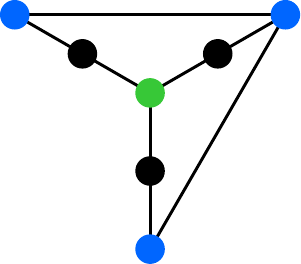}
        \caption{Long Claw plus two edges}
        \label{fig:LC2}
    \end{subfigure}
    \hfill
    \begin{subfigure}[b]{0.3\textwidth}
        \centering
        \includegraphics[width=0.45\textwidth]{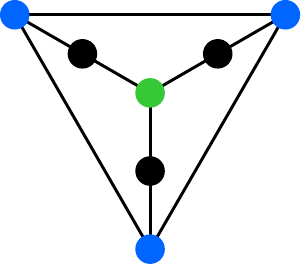}
        \caption{Long Claw plus three edges}
        \label{fig:LC3}
    \end{subfigure}
    \caption{Additional forbidden induced subgraphs for $r$-acyclic circular-arc graphs in the case $r = 4$.}
    \label{fig:Forb-Ind-Sub-4-ACA}
\end{figure}

We are now ready to prove \cref{main:thm:Char-Forb-Ind-Sub-r-ACA}. 

\begin{proof}[Proof of \cref{main:thm:Char-Forb-Ind-Sub-r-ACA}]
    We prove \cref{item:r-ACA:locCov}$\to$\cref{item:r-ACA:locInt}$\to$\cref{item:r-ACA:ForbInd}$\to$\cref{item:r-ACA:locCov}, \cref{item:r-ACA:ForbInd}$\to$\cref{item:r-ACA:r-ACA}, and \cref{item:r-ACA:r-ACA}$\to$\cref{item:r-ACA:ForbInd}.
    \begin{claim}
        \cref{item:r-ACA:locCov} implies \cref{item:r-ACA:locInt}.
    \end{claim}

    \begin{claimproof}
        By \cref{item:r-ACA:locCov}, the $r$-local cover $G_r$ is interval, and thus \cref{lem:IntervalIsLocallyInterval} yields that $G_r$ is $r$-locally interval. As $p_r$ preserves $r/2$-balls, it follows that $G$ is $r$-locally interval, as desired.
    \end{claimproof}
    \begin{claim}\label{claim:ForbIndSubAreSmall}
        For every graph $F$ depicted in \cref{fig:Ford-Ind-Sub-Interval}, there is a vertex $v_F$ of $F$ such that $F \subseteq B_{F}(v_F, 4/2)$.
        For every graph $F$ depicted in \cref{fig:Forb-Ind-Sub-4-ACA}, there is a vertex $v_F$ such that $F \cap B_{F}(v_F,4/2)$ is the long claw.
    \end{claim}

    \begin{claimproof} 
        In both cases, let $v_F$ be the green vertex in the respective figure. The blue vertices are the ones with distance $2$ to $v_F$.
    \end{claimproof}

    \begin{claim}
        \cref{item:r-ACA:locInt} implies \cref{item:r-ACA:ForbInd}.
    \end{claim}

    \begin{claimproof}
        Since $G$ is $r$-locally interval, it holds that $G$ is $r$-locally chordal. Thus, by \cref{thm:locally-chordal}, $G$ is wheel-free and $r$-chordal.
        
        Suppose for a contradiction that $G$ contains a graph from \cref{fig:Ford-Ind-Sub-Interval} as an induced subgraph $F$. 
        Then, $F$ is contained in $B_F(v_F,4/2) \subseteq B_G(v_F,4/2)$ by \cref{claim:ForbIndSubAreSmall}.
        Since $F$ is an induced subgraph of $G$ and $r \geq 4$, the graph $F$ is an induced subgraph of the interval graph $B_G(v_F,r/2)$, which is a contradiction to \cref{thm:Char-Forb-Ind-Sub-Interval}.
        This completes the proof of the claim for $r \geq 5$. 

        Now assume $r = 4$.
        Suppose for a contradiction that $G$ contains a graph from \cref{fig:Forb-Ind-Sub-4-ACA} as an induced subgraph $F$. Then, there is a vertex $v_F$ of $F$ such that $F \cap B_{F}(v_F,4/2)$ is the long claw depicted in \cref{fig:Ford-Ind-Sub-Interval}.
        Since $F$ is an induced subgraph of $G$ and $v_F$ is a vertex of $G$, it follows that $F \cap B_{F}(v_F,4/2) = F \cap B_{G}(v_F,4/2)$, and thus the long claw is an induced subgraph of $B_{G}(v_F,4/2)$, which contradicts \cref{thm:Char-Forb-Ind-Sub-Interval}. This completes the proof of the claim for $r= 4$.
    \end{claimproof}

    \begin{claim}
        \cref{item:r-ACA:ForbInd} implies \cref{item:r-ACA:locCov}.
    \end{claim}

    \begin{claimproof}
        By \cref{thm:Char-Forb-Ind-Sub-Interval}, it suffices to show that the $r$-local cover $G_r$ has no graph $F$ depicted in \cref{fig:Ford-Ind-Sub-Interval} as an induced subgraph.
        Suppose for a contradiction that $G_r$ contains a graph depicted in \cref{fig:Ford-Ind-Sub-Interval} as an induced subgraph $F$.
        Then $F$ is contained in $B_F(v_F,4/2) \subseteq B_{G_r}(v_F,4/2)$ by \cref{claim:ForbIndSubAreSmall}.

        First, assume that $r \geq 5$.
        Since $F$ is an induced subgraph of $G_r$ and $r \geq 5$, the graph $F$ is an induced subgraph of the subgraph $B_{G_r}(v_F,5/2)$ of $G_r$.
        As $p_r$ preserves $r/2$-balls and $r \geq 5$, the graph $F$ is also an induced subgraph of the induced subgraph $B_{G}(p_r(v_F), 5/2)$ of $G$, so $F$ is an induced subgraph of $G$, which contradicts that $F$ is not an induced subgraph of $G$ by \cref{item:r-ACA:ForbInd}.

        Now, assume that $r = 4$.
        Since $F$ is an induced subgraph of $G_r$ and $r = 4$, the graph $F$ is an induced subgraph of the subgraph $B_{G_r}(v_F,4/2)$ of $G_r$.
        As $p_r$ preserves $r/2$-balls and $r = 4$, the graph $F$ is also an induced subgraph of $B_{G}(p_r(v_F), 4/2)$.
        Then only the blue vertices of $F$ in \cref{fig:Ford-Ind-Sub-Interval} have distance $2$ to $p_r(v_F)$ in $G$. 
        If $F$ is not the long claw, then it is a simple check that every potential edge of $G$ between some pair of the blue vertices yields an induced cycle of length $4$ in $G$, which contradicts that $G$ is in particular $4$-chordal by \cref{item:r-ACA:ForbInd}.
        Thus assume that $F$ is the long claw.
        Then, the addition of the potential edge of $G$ between pairs of blue vertices in \cref{fig:LongClaw} gives one of the two graphs depicted in \cref{fig:Forb-Ind-Sub-4-ACA}, which then is an induced subgraph of $G$, contradicting \cref{item:r-ACA:ForbInd}.
    \end{claimproof}

    \begin{claim}
        \cref{item:r-ACA:ForbInd} implies \cref{item:r-ACA:r-ACA}.
    \end{claim}

    \begin{claimproof}
        First, we note that it suffices to show that $G$ has none of the graphs depicted in \cref{fig:Ford-Ind-Sub-NHCA} as induced subgraphs. 
        Suppose $G$ forbids each graph in \cref{fig:Ford-Ind-Sub-NHCA}. By \cref{item:r-ACA:ForbInd}, $G$ forbids each graph in \cref{fig:Ford-Ind-Sub-Interval}. Then $G$ is a 3-acyclic circular-arc graph by \cref{thm:Char-Forb-Ind-Sub-N-H-CA}. 
        As $G$ is $r$-locally chordal by \cref{item:r-ACA:ForbInd} and \cref{thm:locally-chordal}, it follows that $G$ is in fact $r$-acyclic circular-arc by \cref{thm:AK-thm5}. Therefore, it is sufficient to show that $G$ forbids each graph in \cref{fig:Ford-Ind-Sub-NHCA} as an induced subgraph.

        Every graph $F$ depicted in \cref{fig:Ford-Ind-Sub-NHCA} except the one in \ref{fig:1-6} contains an induced cycle of length $4$, and thus is not an induced subgraph of $G$ by \cref{item:r-ACA:ForbInd}.
        It remains to show that the graph $F$ depicted in \cref{fig:1-6} is not an induced subgraph of $G$. Since $F$ contains an induced cycle of length 5, it follows from \cref{item:r-ACA:ForbInd} that $G$ excludes $F$ as an induced subgraph when $r \geq 5$. 
        If $r = 4$, then the induced subgraph $F$ of $G$ is precisely the graph depicted in \cref{fig:LC2}, which is a contradiction to \cref{item:r-ACA:ForbInd}.
    \end{claimproof}

    \begin{claim}
        \cref{item:r-ACA:r-ACA} implies \cref{item:r-ACA:ForbInd}.
    \end{claim}

    \begin{claimproof}
        By \cref{thm:Char-Forb-Ind-Sub-N-H-CA} and $r \geq 4 > 3$, an $r$-acyclic circular-arc graph $G$ contains none of the graphs depicted in \cref{fig:Ford-Ind-Sub-Interval} as induced subgraphs.
        This completes the proof for $r \geq 5$.
        For $r = 4$, it suffices to observe that none of the graphs in \cref{fig:Forb-Ind-Sub-4-ACA} is circular-arc.
    \end{claimproof}
    
    This completes the proof of \cref{main:thm:Char-Forb-Ind-Sub-r-ACA}.
\end{proof}

\section{Characterization of 3-locally chordal graphs}

Now, we revisit the case of \cref{conj:loc-I-is-CAG} where $r \leq 3$. For $r \leq 2$, \cref{conj:loc-I-is-CAG} is trivially wrong: every graph is $0$-, $1$-, and $2$-locally interval, since the relevant balls are stars and stars are interval graphs, but there are graphs, e.g.\ the long claw depicted in \cref{fig:LongClaw}, which are not circular-arc.

For $r = 3$, \cref{conj:loc-I-is-CAG} fails for the same reason: the long claw is still $3$-locally interval but not a circular-arc graph. However, every 3-acyclic circular-arc graph is 3-locally interval by definition. 
One may characterize $3$-locally interval graphs by forbidden induced subgraphs, as follows. A vertex $v$ in a graph $G$ is an \defn{apex vertex} if $v$ is adjacent to every other vertex in $G$. 

\begin{mainresult}\label{prop:ForbInducedSubgraphLocInt}
    Let $G$ be a finite graph. Then the following are equivalent:
    \begin{enumerate}
        \item $G$ is $3$-locally interval.
        \item $G$ is wheel-free (equivalently: $3$-locally chordal) and contains none of the graphs depicted in \cref{fig:Ford-Ind-Sub-Interval} with an apex vertex added as an induced subgraph.
    \end{enumerate}
\end{mainresult}

\begin{proof}
    Every $3$-locally interval graph is $3$-locally chordal, as interval graphs are chordal. It is easy to check that each graph in \cref{fig:Ford-Ind-Sub-Interval} with an apex vertex added is not 3-locally interval. 
    Conversely, let $G$ be a $3$-locally chordal graph such that $G$ contains none of the graphs depicted in \cref{fig:Ford-Ind-Sub-Interval} with an added apex vertex as induced subgraph.
    Consider a vertex $v$ of $G$ and the chordal $3/2$-ball $B \coloneqq B_{G}(v, 3/2)$ around $v$.
    By \cref{thm:Char-Forb-Ind-Sub-Interval}, it suffices to show that $B$ contains none of the graphs depicted in \cref{fig:Ford-Ind-Sub-Interval} as an induced subgraph.
    Suppose for a contradiction that $B$ contains a graph from \cref{fig:Ford-Ind-Sub-Interval} as an induced subgraph $B'$.
    Since $v$ is adjacent to all other vertices in $B$ but no vertex in the graphs depicted in \cref{fig:Ford-Ind-Sub-Interval} is adjacent to all other vertices in that graph, $v$ is not a vertex of $B'$. Thus, the subgraph induced on the vertex set of $B'$ together with $v$ is one of the induced subgraphs forbidden in $G$.
\end{proof}

\printbibliography

\end{document}